\documentclass[]{article}
\usepackage[margin=1in]{geometry}
\usepackage{amsthm, amssymb}
\usepackage{amsmath}
\usepackage{graphicx}
\usepackage{float}
\DeclareRobustCommand{\stirling}{\genfrac[]{0pt}{}}
\DeclareRobustCommand{\sterling}{\genfrac\{\}{0pt}{}}

\newtheorem{theorem}{Theorem}
\newtheorem{lemma}{Lemma}

\title{More Absent-Minded Passengers}

\author{Jacob A. Boswell\footnote{Mathematics Department, Missouri
    Southern State University, e-mail: \texttt{boswell-j@mssu.edu}},
  Jacob N. Clark\footnote{Mathematics Department, Missouri Southern
    State University, e-mail: \texttt{clark-jacob@mssu.edu}}, and Chip
  Curtis\footnote{Mathematics Department, Missouri Southern State
    University, e-mail: \texttt{curtis-c@mssu.edu}}}

\begin{document}

\maketitle

\let\thefootnote\relax
\footnotetext{MSC: Primary 60C05, Secondary 00A08}
\addtocounter{footnote}{-2}\let\thefootnote\svthefootnote

\begin{abstract}
We offer a formula for the probability distribution of the number of misseated
airplane passengers resulting from the presence of multiple
absent-minded passengers, given the number of seats available and the
number of absent-minded passengers. This extends the work of Henze and
Last on the absent-minded passenger problem.
\end{abstract}

\section{Introduction}
A recent article by Henze and
Last, \textit{Absent-Minded Passengers} \cite{Henze19}, considers the
problem of $k$ absent-minded passengers on an airplane with $n$
passengers assigned to $n$ seats. The absent-minded passengers are
assigned seats $\{1,2,...,k\}$, with the other passengers assigned
seats $\{k+1,...,n\}$. The passengers are seated in order of passenger
number. When it is time for one of the absent-minded passengers to
choose a seat, that passenger chooses an unoccupied seat at random,
with an equal likelihood for each of the unoccupied seats. When it is
time for a non-absent-minded passenger to choose a seat, that
passenger sits where assigned, if the assigned seat is available,
otherwise choosing an unoccupied seat at random. The authors of
\cite{Henze19} determine the probability distribution in the case
 where \(k\), the number of misseated passengers, is one, as well as the expected
value and variance for all \(k\geq 1\). In this paper, we find the
probability distribution for all positive integers $k$. 

We claim that, with $n$ passengers, the first $k$ of whom are
absent-minded, the probability that exactly $m$ of them will be
misseated is given by the following result.

\begin{theorem}[Main Result] \label{thm:mainresult} The probability of $m$ misseated passengers  is
\begin{equation*}
P_{n,k}(m)
=\frac{(-1)^m(n-k)!}{n!} {k \choose m}
+\frac{1}{n!}\sum_{s=1}^k \stirling{n-k+1}{m-s+1} {k \choose s} s! 
  \sum_{\ell=1}^s  \frac{(-1)^{s-\ell}\ell^{m-s}}{(s-\ell)!}.
\end{equation*}
\end{theorem}

Here, $\stirling{i}{j}$ is the unsigned Stirling number of the first kind, which is the number of permutations of $i$ elements with $j$ disjoint cycles, with the convention that $\stirling{p}{0}~=0$ and $\stirling{p}{-q}~=~0$ for positive $p$ and $q$ \cite[page 259]{Concrete}. The formula includes the assertion that the probability of exactly one misseated passenger is 0. 

For \(k=1, 2,\) and \(3\) missseated passengers this gives, respectively,
\begin{align*}
P_{n,1}(m) &= \frac{1}{n!}\stirling{n}{m}, \text{ for }m\geq 2 \\
P_{n,2}(m) &= \frac{(-1)^{m}}{n(n-1)}{2\choose m}+\frac{1}{n!}\left (2\stirling{n-1}{m}+\left (2^{m-1}-2\right )\stirling{n-1}{m-1} \right ) 
 \\
P_{n,3}(m) &= \frac{(-1)^{m}}{n(n-1)(n-2)}{3\choose m}+\frac{1}{n!}\left (3\stirling{n-2}{m}\right.\\&\phantom{==}+3\left (2^{m-1}-2\right)\stirling{n-1}{m-1}+\left(2\cdot 3^{m-2}-3\cdot 2^{m-2} +3 \right )\stirling{n-2}{m-2} \Bigg).
\end{align*}

\section{How the passengers can be misseated}

In preparation for the proof of the theorem, we prove the following lemma.

\begin{lemma}
\[
\sum_{k<i_{1}<i_{2}<\cdots < i_{m-s}\leq n}\left (\prod_{j=1}^{m-s}\frac{1}{n-\left (i_{j}-1 \right )} \right )
=\frac{1}{(n-k)!}\stirling{n-k+1}{m-s+1}.
\]
\end{lemma}
\begin{proof}
To prove this, set $\ell_{j}=n-\left(i_{j}-1\right)$. Then the original sum becomes
\[
\frac{1}{(n-k)!}   \sum_{1\leq \ell_1<\ell_2<\cdots <\ell_{m-s}\leq n-k} \frac{(n-k)!}{\ell_1\ell_2\cdots
             \ell_{m-s}}.
\]
For a fixed positive integer $N$, let $g_{N}(x)$ be the generating function of the Stirling numbers of the first kind \cite[page 263]{Concrete}; that is,
\[
g_{N}(x) =x(x+1)\cdots (x+N-1)=\sum_{i=0}^{N}\stirling{N}{i}x^{i}. 
\]
By equating coefficients  of \(x^i\) in this equation, we find that 
\[
\stirling{N}{i} = \sum_{0\leq a_1<a_2<\cdots <a_{N-i}< N} a_1 a_2 \cdots a_{N-i},
\]
and therefore
\begin{align*}
  \frac{1}{(n-k)!}   \sum_{1\leq \ell_1<\ell_2<\cdots <\ell_{m-s}\leq n-k} \frac{(n-k)!}{\ell_1\ell_2\cdots
             \ell_{m-s}} &= \frac{1}{(n-k)!} \stirling{n-k+1}{m-s+1}.
\end{align*}
\end{proof}

Before proving the main theorem, we first prove the formula below. We later simplify this result to give Theorem~\ref{thm:mainresult}. 

\begin{theorem} \label{thm:unsimplified}
\begin{align*} 
  {P}_{n,k}(m) & = \frac{1}{n!}\sum_{s=0}^{{{k}}}{k \choose
                      s}\sum_{t=0}^{s }(t!)^2  \sterling{m-s}{t} \stirling{n-k+1}{m-s+1}
                       \sum_{r=t}^{s} {s\choose r} L(r,t) 
                      d_{s-r}.
\end{align*}
\end{theorem}

Here, $\sterling{i}{j}$ is the Stirling number of the second kind, which counts the number of ways to partition a set of $i$ labeled objects into $j$ nonempty unlabeled subsets \cite[page 258]{Concrete}; $L(i,j)$ is the Lah number, which counts the number of ways a set of $i$ elements can be partitioned into $j$ nonempty linearly-ordered subsets \cite{Lah55,PePi07}; and $d_i$ is the number of derangements of a set of $i$ elements, that is, the number of permutations with no fixed points \cite[page 194]{Concrete}. Following \cite[pages 262]{Concrete}, we adopt the following conventions for positive integers $p$ and $q$:
\[
\sterling{-p}{q}=0,\sterling{-p}{0}=0,\text{ }\sterling{0}{q}=0,\text{ }\sterling{0}{0}=1, \stirling{p}{0}=0\text{ and }\stirling{p}{-q}=0.
\]

\begin{proof}
Since the absent-minded passengers are those with the lowest numbers, we associate them with the first-class cabin and the non-absent-minded passengers with the main cabin. The probability that exactly $m$ passengers are misseated is the sum over $s$ of the probabilities that a total of exactly $m$ passengers, including $s$ from first class and $m-s$ from the main cabin, are misseated.

The probability of a specific arrangement of the $k$ first-class passengers is
\[
  \frac{1}{n}\cdot \frac{1}{n-1}\cdots \frac{1}{n-(k-1)} = \frac{(n-k)!}{n!},
\]
and the probability of a specific sequence
\(
  i_1 < i_2 < \cdots < i_{m-s}
\)
of missseated main cabin passengers is given by
\[
\prod_{j=1}^{m-s}
  \frac{1}{n-(i_j-1)},
\]
since when it is time for passenger $i_{j}$ to be seated, there are $n-\left (i_{j}-1\right )$ seats available.

The total probability of the outcome is thus
\[
  \frac{(n-k)!}{n!} \cdot\prod_{j=1}^{m-s}
  \frac{1}{n-(i_j-1)}.
\]
We now count the number of outcomes with exactly $m$ misseated passengers including exactly $s$ first-class passengers and the particular passengers $i_{1}<i_{2}<\cdots i_{m-s}$ from the main cabin. There are \({k \choose s}\) ways of choosing which first-class passengers are misseated. 

The misseating of main cabin passengers  \(i_1, i_2, \ldots, i_{m-s}\)
occurs in \emph{threads}, with a thread consisting of a non-empty sequence of first-class passengers followed by a non-empty sequence of main cabin passengers. The number of threads is at least zero (in the case that no main-cabin passengers are misseated) and at most \(s\).
For a given number $t$ of threads,  at least \(t\) and at most \(s\) of the misseated first-class passengers are elements of these threads. Let the number of these absent-minded passengers be \(r\). There are then 
\(s-r\) misseated first-class passengers who are not part of a thread. 

\begin{figure}
\centering
\includegraphics[scale=1.05]{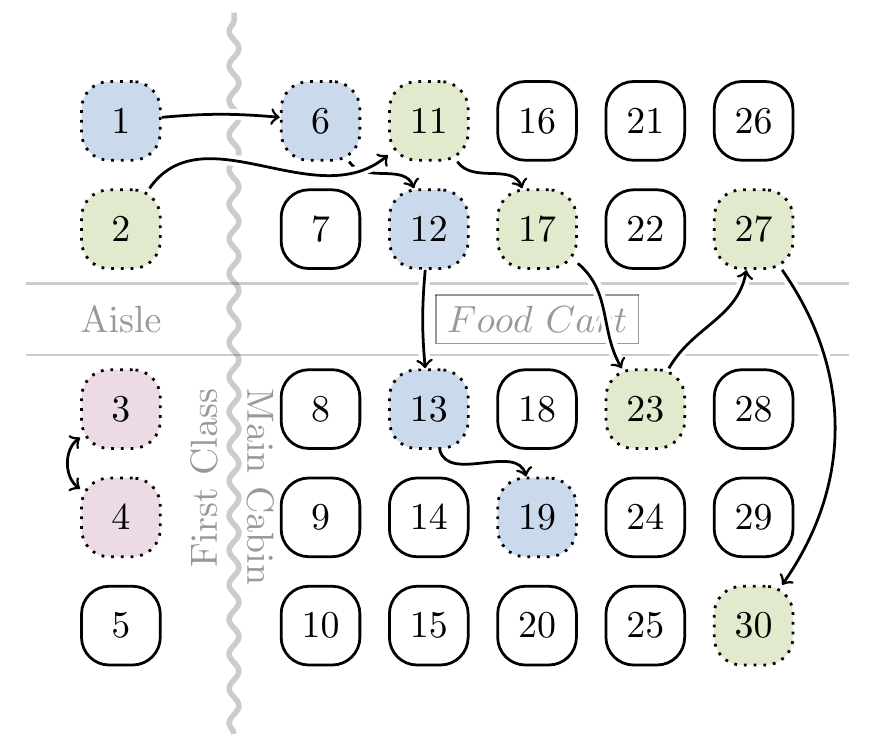}
\caption{Airplane passengers are misseated in threads. Here, \(n=30\), \(k=5\) and \(m=13\). Furthermore, \(s=4\) and \(r=2\).
The threads terminate when passenger \(19\) sits in either seat \(1\) or \(2\). Passenger \(30\) must then sit in whichever of these two seats remain.}\label{fig:threads}
\end{figure}

There are \({s \choose r}\) choices for the \(r\) first-class passengers who are in threads. These \(r\) passengers can be placed into \(t\) threads in \(L(r,t)\) ways. 
The \(i_1, \ldots, i_{m-s}\) passengers can be placed into these \(t\) threads in
\((t!)\sterling{m-s}{t}\)
ways.

Each thread ends with a main cabin passenger sitting in the seat of a first-class passenger who is seated first in a thread. This can hapen in \(t!\) ways. The remaining \(s-r\) misseated passengers
permute their seats, with none fixed. This can happen in $d_{s-r}$ ways. A visualization of this can be seen in Figure~\ref{fig:threads}. Thus,

\begin{align*}
  \mathbb{P}&(m \text{ misseated, including the main-cabin passengers
  } i_1, i_2, \ldots, i_{m-s} ) \notag \\
  &=\left(\sum_{t=0}^{s} \sum_{r=t}^{s} {k\choose s}{s\choose r}
    L(r,t) (t!)^2 \sterling{m-s}{t}d_{s-r}\right)\notag 
  \frac{(n-k)!}{n!}
  \prod_{j=1}^{m-s} \frac{1}{n-(i_j-1)},
\end{align*}
and
\begin{align*}
  \mathbb{P}&(m \text{ misseated, including } s \text{
    first-class passengers}) \\
   & = \left(\sum_{t=0}^{s} \sum_{r=t}^{s}
    {k \choose s}{s\choose r} L(r,t) (t!)^2
    \sterling{m-s}{t}d_{s-r}\right) 
  \cdot \frac{(n-k)!}{n!}  \sum_{k<i_1<i_2<\cdots<i_{m-n}} \left(
    \prod_{j=1}^{m-s} \frac{1}{n-(i_{j}-1)}\right) \\
   &= \frac{1}{n!}\left(\sum_{t=0}^{s} \sum_{r=t}^{s}
    {k \choose s}{s\choose r} L(r,t) (t!)^2
    \sterling{m-s}{t}d_{s-r}\right)  \stirling{n-k+1}{m-s+1}.
\end{align*}

Summing over $s$ gives
\begin{align*}
  {P}_{n,k}(m) 
                    &= \frac{1}{n!}\sum_{s=0}^{k}{k \choose
                      s}\stirling{n-k+1}{m-s+1}
\sum_{t=0}^{s }(t!)^2 \sterling{m-s}{t} \sum_{r=t}^{s} {s\choose r} L(r,t) 
                      d_{s-r},
\end{align*}
as claimed.
\end{proof}

\section{Proof of main result}

We proceed to obtain Theorem~\ref{thm:mainresult} from Theorem~\ref{thm:unsimplified}. To do so, we begin with the sum over $r$ using formulas for the Lah numbers \cite{PePi07} and the derangements \cite[page 195]{Concrete}. For $t \geq 1$, we have

\begin{align}\label{eqn:resultabove}
\sum_{r=t}^{s} {s\choose r} L(r,t) d_{s-r} 
& = \sum_{r=t}^{s} {s\choose r} {r-1\choose t-1}\frac{r!}{t!}
(s-r)!\sum_{j=0}^{s-r}\frac{(-1)^{j}}{j!}\nonumber \\
& =  \frac{s!}{t!}\sum_{j=0}^{s-t}\frac{(-1)^{j}}{j!}
\sum_{r=t}^{s-j}{r-1\choose t-1}\nonumber \\
& =  \frac{s!}{t!}\sum_{j=0}^{s-t}\frac{(-1)^{j}}{j!}{s-j\choose t}.
\end{align}
We note that if $t=0$, then $\sum_{r=t}^{s} {s\choose r} L(r,t) d_{s-r} $ and $ \frac{s!}{t!}\sum_{j=0}^{s-t}\frac{(-1)^{j}}{j!}{s-j\choose t}$ both equal $d_s$, so we can use the result of the above calculation in that case as well.

The following result is simple but useful. We record it as a lemma.

\begin{lemma}
For positive integers $J$, $K$, $L$, with $L\leq K$,
\[
\sum_{J=L}^{K}(-1)^{J}{K-L\choose J-L}
=(-1)^{L}\delta_{L,K},
\]
where $\delta_{L,K}$ is \(1\) if $L=K$ and \(0\) otherwise.
\end{lemma}
\begin{proof}
Make the change of variables $I=J-L$ to get
\[
(-1)^{L}\sum_{I=0}^{K-L}(-1)^{I}{K-L\choose I};
\]
the sum is the expansion of $(1-1)^{K-L}$, which is 0 unless $L=K$.
\end{proof}

We now consider the sum over $t$ in the equation of Theorem \ref{thm:unsimplified}, substituting the result obtained in equation~(\ref{eqn:resultabove}) above. 
For $s < m$, a formula for the Stirling numbers of the second kind \cite[page 265]{Concrete} gives

\begin{align*}
\sum_{t=0}^{s }(t!)^2 \sterling{m-s}{t}
\sum_{r=t}^{s} {s\choose r} L(r,t) d_{s-r} 
& = (s!)\sum_{t=0}^{s }\sum_{\ell=0}^{t}(-1)^{t-\ell}{t\choose \ell}\ell^{m-s}\sum_{j=0}^{s-t}\frac{(-1)^{j}}{j!}{s-j\choose t} \\
& = (s!)\sum_{\ell=0}^{s}(-1)^{\ell}\ell^{m-s}\sum_{j=0}^{s-\ell}\frac{(-1)^{j}}{j!}\sum_{t=\ell}^{s-j}(-1)^{t}{t\choose \ell}{s-j\choose t}. \end{align*}
Using trinomial revision \cite[page 174]{Concrete} gives
\[
{t\choose \ell}{s-j\choose t}={s-j\choose \ell}{s-j-\ell\choose s-j-t}
={s-j\choose \ell}{s-j-\ell\choose t-\ell},
\]
so that the above becomes
\begin{align*}
& = (s!)\sum_{\ell=0}^{s}(-1)^{\ell}\ell^{m-s}\sum_{j=0}^{s-\ell}\frac{(-1)^{j}}{j!}{s-j\choose \ell}\sum_{t=\ell}^{s-j}(-1)^{t}{s-j-\ell\choose t-\ell} \\
& = (s!)\sum_{\ell=0}^{s}\ell^{m-s}\sum_{j=0}^{s-\ell}\frac{(-1)^{j}}{j!}{s-j\choose \ell}\delta_{s-j-\ell,0} \\
& = (s!)(-1)^{s}\sum_{\ell=0}^{s}(-1)^{\ell}\frac{\ell^{m-s}}{(s-\ell)!}.
\end{align*}

We now address the case $s=m$. We have
\begin{equation*}
\sum_{t=0}^{s }(t!)^2 \sterling{m-s}{t}  \frac{s!}{t!}\sum_{j=0}^{s-t}\frac{(-1)^{j}}{j!}{s-j\choose t} = \sum_{t=0}^{m}(t!)^2 \sterling{0}{t}  \frac{m!}{t!}\sum_{j=0}^{m-t}\frac{(-1)^{j}}{j!}{m-j\choose t}.
\end{equation*}
The above is
\begin{align*}
\sum_{t=0}^{m}&(t!)\delta_{t,0} m!\sum_{j=0}^{m-t}\frac{(-1)^{j}}{j!}{m-j\choose t}  
 = m!\sum_{j=0}^{m}\frac{(-1)^{j}}{j!} 
  = m!(-1)^{m}\sum_{\ell=0}^{m}\frac{(-1)^{\ell}\ell^{0}}{(m-\ell)!}. 
\end{align*}

Substituting in the original equation now gives
\[
 {P}_{n,k}(m) = \frac{1}{n!}\sum_{s=0}^k \stirling{n-k+1}{m-s+1} {k \choose s} (s!)(-1)^s 
 \left[ \frac{\delta_{s,m}}{m!} + \sum_{\ell=1}^s \frac{ (-1)^\ell\ell^{m-s}}{(s-\ell)!}\right].
\]
Interpreting ${k\choose m}$ as 0 when $k<m$, and noting that $\stirling{n-k+1}{1}=(n-k)!$, we can rewrite this last result as
\begin{equation*}
{P}_{n,k}(m) = \frac{(-1)^m(n-k)!}{n!} {k \choose m} 
+\frac{1}{n!}\sum_{s=0}^k \stirling{n-k+1}{m-s+1} {k \choose s} s! 
  \sum_{\ell=1}^s  \frac{(-1)^{s-\ell}\ell^{m-s}}{(s-\ell)!},
\end{equation*}
as required. This proves Theorem~\ref{thm:mainresult}.

For a visual interpretation of this function for several \(k\) when the number of
passengers, \(n\), is \(100\), we direct the reader
to Figure~\ref{fig:histograms}.

\begin{figure}[t]
\centering
\includegraphics{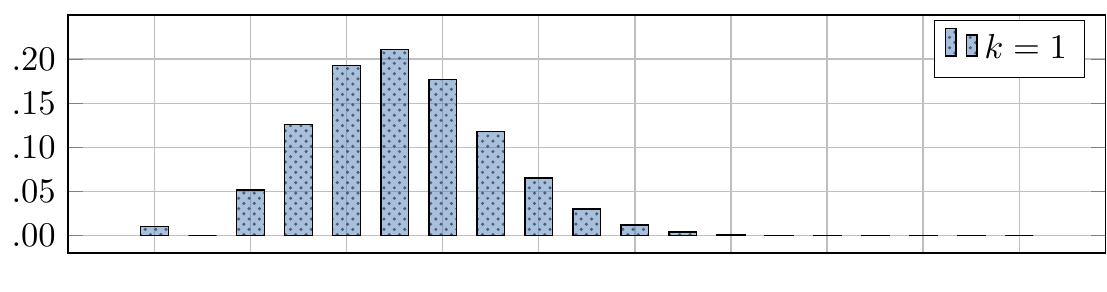}
\includegraphics{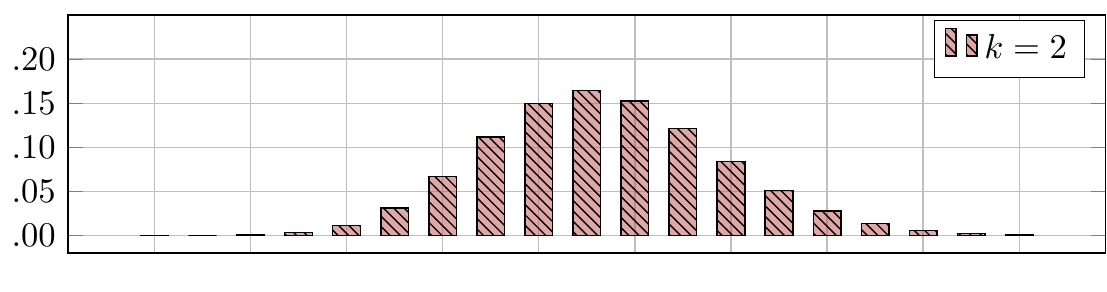}
\includegraphics{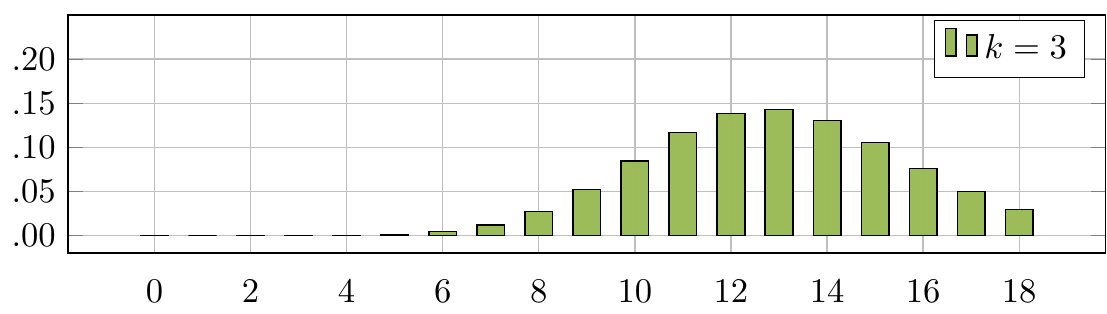}
\caption{A graph of the probability as a function of \(m\) given by
  our formula \(P_{n,k}(m)\), for an \(n=100\) passenger plane, with
  \(k=1, 2\) and \(3\) absent-minded passengers.} \label{fig:histograms}
\centering
\end{figure}
\pagebreak


\bibliography{references}
\bibliographystyle{plain}

\end{document}